\documentclass[11pt]{article}
\usepackage{amsmath,amsthm,amssymb,cite}
    \newtheorem{theorem}{Theorem}[section]
    \newtheorem{proposition}[theorem]{Proposition}
    \newtheorem{lemma}[theorem]{Lemma}
    \newtheorem{corollary}[theorem]{Corollary}

    \numberwithin{equation}{section}
    \numberwithin{theorem}{section}
    
    \renewcommand{\(}{\left(} 
    
    \newcommand{\nn}{\nonumber}
    \newcommand{\bb}[1]{{\mathbb #1}}

    \renewcommand{\epsilon}{\varepsilon}

    \newcommand{\id}{{1 \mskip -5mu {\rm I}}}
    
    \newcommand{\varsh}{\mathop{\rm sinh}\nolimits}
    \newcommand{\varch}{\mathop{\rm cosh}\nolimits}

    \newcommand{\vararch}{\mathop{\rm arcosh}\nolimits}

    \newcommand{\e}{\mbox{e}}
    
\begin{document}
    
    \title {On the density of the winding number of planar Brownian motion}
    \author {Stella Brassesco 
    \and Silvana C. Garc\'\i a Pire
    }
     \date{}
     \maketitle

    \begin{abstract}
    We obtain a formula for the density  $f(\theta, t)$ of the winding number
 of a  planar Brownian motion $Z_t$ around the origin.
 From this formula we deduce an expansion for  $f(\log(\sqrt t)\,\theta\,,\,t)$
 in inverse powers of $\log \sqrt {t}$ and  $(1+\theta^2)^{1/2}$  which in particular yields
 the corrections of any order to Spitzer's asymptotic law (\cite{spitzer}).
 We also obtain an expansion for $f(\theta,t)$ in inverse powers of $\log \sqrt {t}$, which yields
 precise asymptotics  as $t \to \infty$ for a local limit theorem for the windings.

    \end{abstract}
    
      \section {Introduction}
      
      In his celebrated paper about 2--dimensional  Brownian motion,  F. Spitzer considered the transition
      probabilities  of  $\Theta_t$, its continuous winding number around the origin,
        and showed in particular  that the limiting distribution 
      of $\frac{\Theta_t}{\log \sqrt t}$ is the Cauchy distribution:
      \begin{equation}
      \label{spitzer}
     \underset{t\to \infty}{\lim}\bb P\big(\frac{\Theta_t}{\log \sqrt t}\le \alpha\big)
      =\frac 1\pi \int_{-\infty}^{\alpha}\frac{d\theta}{1+\theta^2}.
      \end{equation}
      
      Several proofs of this result have since then  been given, see for instance page 43 
       of the book \cite{w} by L.G.C. Rogers and D. Williams
      and the discussion and references therein.
       Extensions to more general situations are  known as well,
        for instance those considered by J.Pitman and M. Yor in  \cite{pyor0}, \cite {pyor} and \cite{pyor2}
      in the frame of asymptotic laws  for planar Brownian motion.

    More recently, V. Bentkus, G.Pap and M. Yor  in 
   (\cite{papyor2} and \cite{papyor1}) obtained by Fourier methods  an expansion 
    of the distribution function  of $\Theta_t$, that yields detailed asymptotics,
     both in $t$ and $\alpha$, as they go to $\infty$. 
    
    F. Delbaen, E. Kowalski and A. Nikeghbali proved in \cite{dkn} 
      (also using Fourier methods) the following  local limit theorem for $\Theta_t$: if $\alpha<\beta\in \bb R$,  
     \begin{equation}
      \label{dkn}
 \underset{t\to \infty}{\lim}\log{\sqrt t}\,\bb P\big(\alpha<\Theta_t < \beta \big)
      =\frac {\beta-\alpha}{\pi}.
      \end{equation}

   In the present work,  we obtain a formula for the density $f(\theta, t;\rho)$ of $\Theta _t$
    when the initial condition $Z_0=\rho\neq 0,\rho\in \bb R$, that follows by integrating
    an expression for  the joint density $p(r,\theta, t;\rho)$ of $(|Z_t|\,,\Theta_t)$ that 
     appears in \cite {sb} (and follows as well from results in \cite{yor}, see also the comments there for older references). We compute $f$ in the next section. Let us mention that in  Chapter V of \cite{mansuyor},
      R. Mansuy and M. Yor discuss some representations for the distribution of the winding number
       of the Brownian lace of lengh $t$. 
     
     The expression  for the density  $f$  is given in terms of an integral
including a couple  of fractions. In the last section, we expand 
these fractions in inverse powers of $\log\sqrt t$ and  obtain  two asymptotic expansions,
 after term by term integration.  
    The first one is an expansion for $f(\log{\sqrt t}\,\theta, t;\rho)$  very close to that given in 
    \cite{papyor1}. The difference is that the dependence  of the coefficients on $\theta$ is
    given more explicitly,  and that  
     we obtain also accurate  estimates on the behaviour in $t$ of those coefficients.
       A more precise comparison is discussed in  Remark 4  
    after the statement of the corresponding result (Theorem \ref{th1}). 
    The second expansion is for  $f(\theta, t;1)$ in inverse powers of $\log{\sqrt t}$,
     that yields  in particular the corrections of any order to \eqref{dkn}.
      The  result is stated as Theorem \ref{th2}.

       \bigskip

       \section { A formula for the density }
       We consider a planar Brownian motion $Z_t$ starting at $z_0\neq 0$, 
       $\Theta_t$ a continuous determination of its argument and 
       $R_t=|Z_t|$. Let us call $p$ the joint  density of   $\Theta_t$ and  
       $R_t$, 
       \begin{equation}
       \nonumber
       p(r,\theta,t;\rho,\alpha) \,rdrd\theta=\bb P(R_t\in dr,
       \Theta_t\in d\theta\big\vert R_0=\rho,\Theta_0=\alpha)
       \end{equation}
       Since $p$ is clearly a function of $\theta-\alpha$,
        it suffices to consider $\alpha=0$, what we do in the sequel.
        The following expression for $p$ was deduced in \cite{sb}
        (see also \cite{yor}), 
        
    \begin{equation}
    \label{ptr}
       p(r,\theta,t;\rho)=\frac{1}{\pi\,t}\exp{(-\frac{r^2+\rho^2}{2t})}
       \int^{\infty}_0\,\cos(\nu\theta)\,I_\nu(\frac{\rho r}{t})\,d\nu
       \end{equation}
       Integration in $r$ of the above formula yields an expression for 
       the density 
       of $\Theta_t$, $f(\theta,t;\rho) \,d\theta=P(
       \Theta_t\in d\theta\big\vert R_0=\rho)$.

       Denote by $ \id _A$ the indicator function of the set $A$. We have then: 
       
       \begin{proposition}
       \label{p1}
        The density $f$  of $\Theta_t$ is given by 
       \begin{align}
\nonumber
       f(\theta,t;\rho) =\,&\frac{\rho \,
       \e^{-\frac{\rho^2}{4 t}(1-\cos 2\theta)}}{\sqrt{2\pi t}}\,
         \cos{\theta}\,
         \id _{(-\pi/2,\pi/2)}(\theta)
         \\\label{f1}
         +\,\frac{\rho \,\e^{-\frac{\rho^2}{4 t}}}{2\pi\,\sqrt{2\pi t}}
         &\int^{\infty}_0 \!
         \e^{-\frac{\rho^2}{4 t}\cosh \omega }
         \sinh(\omega/2) 
         \\ \nonumber 
          &\qquad \times  \big(\frac{\omega/2}{(\omega/2)^2\,+(\theta+\pi/2)^2}
         \,+\,\frac{\omega/2}{(\omega/2)^2\,+(\theta-\pi/2)^2}\big)\,d\omega
         \end{align} 
         \end{proposition}
         \medskip 
         Before proving the previous proposition, let us show another
           formula for  the density $f$,
that will prove more suitable to obtaining asymptotic expansions (as $t\to \infty$). 
It follows directly from the change of  variables $\frac{\rho^2}{4 t}\cosh \omega \to z$
in the integral above. 
         
\begin{corollary} 
The following formula for $f$ holds: 
 \begin{align}
\nonumber
       f(\theta, & t; \, \rho) =\frac{\rho \,
       \e^{-\frac{\rho^2}{4 t}(1-\cos 2\theta)}}{\sqrt{2\pi t}}\,
         \cos{\theta}\,
         \id _{(-\pi/2,\pi/2)}(\theta)
         \\\label{f2} 
           +\,&\frac{1}{2\pi\sqrt\pi}\,\int ^{\infty}_{\rho^2/4t}\,
\frac{\e^{-(z+\rho^2/4t)}}{\sqrt{z+\rho^2/4t}}\,\times
 \\ \nonumber
 &\big(\frac{\frac 12\vararch(\frac{4tz}{\rho^2})}{(\frac 12\vararch(\frac{4tz}{\rho^2}))^2\,+(\theta+\pi/2)^2}
         \,+\,\frac{\frac 12\vararch(\frac{4tz}{\rho^2})}{(\frac 12\vararch(\frac{4tz}{\rho^2}))^2\,+(\theta-\pi/2)^2}\big)\,dz
         \end{align} 
\end{corollary}
\medskip 

\noindent {\it Proof of Proposition \ref{p1}}
Integrate by parts once, recall that $2\,I'_{\nu}(x)=I_{\nu-1}(x)+
          I_{\nu+1}(x)$ and use formula 6.618--4 of \cite{gr} to
           compute
           \begin{multline} 
           \nonumber
           \int^{\infty}_{0}\e^{-\frac{r^2}{2t}}I_{\nu}(\frac{\rho r}{t})\,rdr
           =\rho\, \int^{\infty}_{0}\e^{-\frac{r^2}{2t}}
           I'_{\nu}(\frac{\rho r}{t})\,dr\,=
           \\
           \frac{\rho}{2}
           \int^{\infty}_{0}\e^{-\frac{r^2}{2t}}
           \big[I_{\nu-1}(\frac{\rho r}{t})\,+\,I_{\nu+1}
           (\frac{\rho r}{t})\big]dr
           =\frac{\rho\,\sqrt{2\pi\,t}}{4}\,\e^{\frac{\rho^2}{4t}}
           \big[ I_{\frac{\nu-1}{2}}(\frac{\rho^2}{4t})+
          I_{\frac{\nu+1}{2}}(\frac{\rho^2}{4t})\big]
          \end{multline} 
With the aid of this formula we can integrate \eqref{ptr}, after interchanging the order of the integrals, 
obtaining         

           \begin{align}
\nonumber
           f(\theta,t;\rho)&=\int^{\infty}_{0}p(r,\theta,t;\rho)\,r\,dr 
\\
\label{fourier}&=\,
           \frac{\rho\,\e^{-\frac{\rho^2}{4t}}}{2\sqrt{2\pi\,t}}
           \int^{\infty}_{0}\,\cos(\nu\theta)
           \big[ I_{\frac{\nu-1}{2}}(\frac{\rho^2}{4t})+
          I_{\frac{\nu+1}{2}}(\frac{\rho^2}{4t})\big]\,d\nu
           \end{align} 
Recall next the integral representation  for the Bessel  function $I_\nu$ (formula 8.431--5 of \cite{gr}):
\begin{equation}
           \nonumber
           I_\nu (x)=\frac 1\pi \int_0^{\pi}\!
           \e^{x\cos\omega}\cos(\nu\omega)\,d\omega\,-\,\frac{\sin(\nu\pi)}{\pi}
           \int _0^{\infty}\! \e^{-x\cosh \omega-\nu\omega}\,d\omega
           \end{equation}
Substitution of this last formula in the integral 
in  \eqref{fourier} yields

           \begin{multline} 
\label{s123} \int_0^{\infty}\cos(\nu\theta)
           \big[ I_{\frac{\nu-1}{2}}(\frac{\rho^2}{4t})+
          I_{\frac{\nu+1}{2}}(\frac{\rho^2}{4t})\big]\,d\nu\,=
          \\
          \frac 1\pi\int^{\infty}_0\!\cos(\nu\theta)\int^\pi_0
         \e^{\frac{\rho^2}{4t}\cos\omega}\,\big[\cos((\nu-1)\omega/2)+
          \cos((\nu+1)\omega/2)\big]\,d\omega\,d\nu 
          \\
          -\,\frac 1\pi \int^{\infty}_0\!\cos(\nu\theta)
           \sin(\pi (\nu-1)/2)
          \int^{\infty}_0 \!\e^{-\frac{\rho^2}{4t} 
          \cosh(\omega)-(\nu-1)\omega/2}\,d\omega\,d\nu
          \\
          -\,\frac 1\pi \int^{\infty}_0\!\cos(\nu\theta)
           \sin(\pi (\nu+1)/2)
          \int^{\infty}_0 \!
           \e^{-\frac{\rho^2}{4t} \cosh(\omega)-(\nu+1)\omega/2}\,d\omega\,d\nu \\
    =S_1\,+\,S_2\,+\,S_3
          \end{multline}
To compute the first integral $S_1$, use that $\cos((\nu-1)\omega/2)+
\cos((\nu+1)\omega/2)=2\cos(\nu\omega/2)\cos(\omega/2)$,
 and  recall that the integrand
  is an even function of $\omega$ and $\nu$ to write:  
\begin{multline}
S_1= \frac{1}{2\pi}  \int^\infty_{-\infty} 
 \cos(\nu\theta)\, \int^{\frac{\pi}{2}}_{-\frac{\pi}{2}}\!
2\,\e^{\frac{\rho^2}{4t}\cos 2\omega}
\cos \omega\,\cos(\nu\omega)\,  d\omega d\nu
 \\ \label{s1}
=\, 2\, \e^{\frac{\rho^2}{4t}\cos 2\theta}\cos \theta\,
\id _{(-\frac{\pi}{2},\frac{\pi}{2})}(\theta).
\end{multline} 
The last identity follows from the Fourier inversion theorem, which applies since the inner integral $I(\nu)$
is an $L^1$ function. This in turn can be seen from the fact that both  $|I(\nu)|$ and $|\nu^2 I(\nu)|$ are uniformly bounded. The last statement can be checked integrating by parts twice, using that $2\,\e^{\frac{\rho^2}{4t}\cos 2\omega}\cos \omega$ vanishes at $\pm \frac\pi2$.

To compute $S_2$, change the order of the integrals, recall that\\
$\sin(\pi(\nu-1)/2)=-\cos(\nu\pi/2)$ and integrate
\begin{multline*}
\int^{\infty}_{0}\!\cos{(\nu\theta)}\,\cos{(\nu\pi/2)}\e^{-\frac{\nu y}{2}}\,d\nu=\\
\frac12\big(
\frac{y/2}{(y/2)^2+(\theta+\pi/2)^2}+\frac{y/2}{(y/2)^2+(\theta-\pi/2)^2}\big)
\end{multline*}
The computation of  $S_3$ is similar, and we obtain finally
\begin{align}
\nonumber
S_2 + S_3 = \frac{1}{\pi}
\int^{\infty}_{0}\!&\e^{-\frac{\rho^2}{4t}\cosh \omega}\sinh (\omega/2)
\\ \label{s23}
&\times\Big( \frac{\omega/2}{(\omega/2)^2+(\theta+\pi/2)^2}\,+\,
\frac{\omega/2}{(\omega/2)^2+(\theta-\pi/2)^2}\Big)\,d\omega
\end{align}
From \eqref{s123}, \eqref{s1} and this last equation yield \eqref{f1}.
\qed
\medskip

Despite formulae \eqref{f1} and \eqref{f2} may appear complicated,
we can see that they give information on the behaviour of $\Theta_t$.
In particular, it is not hard to deduce 
for instance  detailed asymptotics as $t \to \infty$. 
We state and prove two precise results in the following section.
 \section { Two asymptotic results }
 \label{s3}

\begin{theorem}
\label{th1}
The density $f(\theta,t;\rho)$ admits the following expansion:
for any natural number $N\ge 0$,  $t>2$ and $\delta\in (0,\frac 12)$,  
\begin{equation}
\label{exp1}
\log (\sqrt t)\,f(\theta\log (\sqrt t),t;\rho)
=\frac{1}{2\pi\sqrt{\pi}}
\sum_{n=0}^{N}\,\frac{(-1)^n\,A_n(\theta)\,C_n(t;\rho)}
{2^{n-1}(1+\theta^2)^{\frac{n+1}{2}}\big(\log (\sqrt t)\big)^n}
\,+\,R_N
\end{equation}
The coefficients satisfy 
\begin{align}
 \label{cnt}
 A_n(\theta)=\frac{\sum_{\underset {k\le n+1}{ k\, even}}
\binom{n+1}{k}(-1)^{\frac k2}\theta^k}
{(1+\theta^2)^{\frac{n+1}{2}}}\,,
\qquad
 C_n(t;\rho)=c_n+ o_n(t^{-\delta})
\end{align}
where the $c_n$ are constants given by  
\begin{equation}
\label{cn}
c_n=
\sum_{\underset {k\le n} {k\, even}}
\!\binom{n}{k}(-1)^{\frac k2}\pi^k\!\int ^{\infty}_{0}\,\frac{\e^{-z}}{\sqrt{z}}\,
\big(\log(8z/\rho^2)\big)^{n-k}\,dz.
\end{equation}

The remainder terms 
$R_N$ satisfy that for some positive constants $k_N$,
\begin{equation}
\label{RN}
|R_N|\le 
\frac{k_N}{(\log(\sqrt t))^{N+1}\,
(1+\theta^2)^{\frac N2+1}}
\end{equation}
\end{theorem}
\bigskip
Let us make some observations  on the previous result. 

\newpage

\noindent {\bf Remarks}
\begin{itemize}
\item[(1)] The dependence on the initial condition $\rho$  is made explicit in the coefficients 
$c_n$, but the constants $k_N$,  and  $o_n(t^{-\delta})$ are not uniform in $\rho$.
 We could have taken $\rho=1$ as usual, but preferred to keep it explicit,
as the dependence may be traced back in order to investigate the behaviour in $\rho$ also.
 \item[(2)] For $N$ even, $A_N\sim \frac{1}{(1+\theta^2)^{1/2}}\mbox{ as } \theta \to \infty$, 
while for $N$ odd,  $A_N\sim 1 \mbox{ as } \theta \to \infty$. Then the order in $\theta$ of the coefficient
$\frac{A_N(\theta)}{(1+\theta^2)^{\frac{N+1}{2}}}$ of 
$\frac{1}{(\log \sqrt t)^N}$ in the expansion is the same as that of $R_N$ for $N$ even. 
\item[(3)] As $c_0=\sqrt \pi$ and $A_0(\theta)=\frac{1}{(1+\theta^2)^{1/2}}$, 
the first term yields precisely Spitzer's result  \eqref{spitzer}. 
The $c_n$ can be expressed in terms of 
derivatives of the Gamma function $\Gamma$, using the formula 
\begin{equation}
\label{Ga}
\int_0^{\infty}\frac{\e^{-z}}{\sqrt z}\big(\log(z)\big )^n\, dz\,=\,\frac{d^n}{dt^n}\Gamma(t)\big|_{t=\frac 12}
\end{equation}
(see for instance formula 4.358.5 on page 578 of \cite {gr}).

\item[(4)]As mentioned in the introduction,  V. Bentkus, G.Pap and M. Yor  in 
(\cite {papyor1} and \cite{papyor2}) obtained an expansion for $f$ (in the case $\rho=1$),
by different methods. 
The expression for the coefficients in $\theta$ in \eqref{exp1} is slightly more
 explicit here than the one they obtain. It is not difficult to see that the first
  three terms coincide with those they compute. 
We  also provide accurate  estimates on the rate of convergence to $0$ of $C_n-c_n$ as $t\to \infty$. 
\end{itemize}
\medskip

As  the density $f$ of $\Theta_t$ is a function of $\frac{\rho}{\sqrt t}$, we  consider $\rho=1$
in the next result, and omit the reference to the initial condition, keeping the notation $f(\theta,t)$.

\begin{theorem}
\label{th2}
For any natural number $N\ge 0$ and $t>2$, the following expansion holds:  
\begin{equation}
\label{exp2}
\log (\sqrt t)\,f(\theta,t)
=\frac{1}{2\pi\sqrt{\pi}}
\sum_{n=0}^{N}\,\frac{g_n(\theta)}{\big(\log (\sqrt t)\big)^n}
\,+\,E_N.
\end{equation}
The  $g_n(\theta)$ are polynomials in $\theta$ given by 
\begin{multline}
\label{gn}
g_n(\theta)=\\(-1)^n
\sum_{\underset{k\le n} {k\, even}}\!\binom{n}{k}(-1)^{\frac k2}
\big((\frac\pi 2+\theta)^k +(\frac\pi 2-\theta)^k\big)\!
\int ^{\infty}_{0}\frac{\e^{-z}}{\sqrt{z}}\,
\big(\log(8z)\big)^{n-k}\,dz
\end{multline}
and the remainder $E_N$ satisfies
\begin{equation} 
\label{En}
\big|E_N\big|\le \frac{S_{N+2}(|\theta|)}{\big(\log (\sqrt t)\big)^{N+1}}.
\end{equation}
$S_{N+2}$ above is a polynomial  of degree $N+2$, whose coefficients may depend on $N$. 
\end{theorem}

\smallskip

\noindent {\bf Remark }
It is easy to compute $g_0=2\sqrt \pi$, and then 
the local limit theorem  \eqref{dkn} follows at once  from Theorem  \ref{th2}. 
Moreover, corrections of any order to \eqref{dkn} in inverse powers of $\log \sqrt t$ follow
from \eqref{exp2}. Simple integration of  the polynomials $g_n$ provide the coefficients, 
let us say $G_n$, for the expansion of 
$\log\sqrt t\, \bb P\big(\alpha<\Theta_t \le \beta \big)=\sum_{n\ge 0}\frac{G_n}{(\log \sqrt t)^n}$. 
For instance, computing with the aid of  \eqref{Ga} and \eqref{gn}, 
 $$g_1=-2\int^\infty_0\frac{\e^{-z}}{\sqrt z} 
 (\log 8+\log z)\,dz=2\sqrt \pi (\gamma+\log \frac 12)\,,
 $$
where $\gamma$ is the Euler constant, we have
$$ 
\log\sqrt t\, \bb P\big(\alpha<\Theta_t \le \beta \big)=
      \frac {\beta-\alpha}{\pi}\,+\,\frac{1}{\log\sqrt t }\,
      \frac {(\beta-\alpha)(\gamma+\log\frac 12)}{\pi} 
      +O\big(1/(\log\sqrt t) \big)^2
      $$
      \bigskip

The proofs of both theorems are given after stating and proving the next three lemmas.

 Let us shorthand
\begin{equation}
\label{defb}
x=\frac{1}{\log\sqrt t}\quad \mbox{ and} \quad 
 b=\frac12\log \big(\frac{4z}{\rho^2}+\sqrt{(\frac{4z}{\rho^2})^2-\frac{1}{t^2}}\big)
\end{equation}
\begin{lemma}
\label{logint}
The density $f$ satisfies: 
\begin{align}
\nonumber
\log (\sqrt t)\,&f(\theta\log (\sqrt t),t;\rho)
\,=\,\mathcal N(\theta,t;\rho)\,+\,
 \frac{1}{2\pi\sqrt\pi}\,\int ^{\infty}_{\rho^2/4t}\,
\frac{\e^{-(z+\rho^2/4t)}}{\sqrt{z+\rho^2/4t}}
\\
& 
\label{lf}
\times\,\Big(\frac{1+bx}{(1+bx)^2+(\theta+\frac{\pi}{2}x)^2}
\,+\,\frac{1+bx}{(1+bx)^2+(\theta-\frac{\pi}{2}x)^2}
\Big)\, dz,
\end{align}

The term $\mathcal N$ is negligible as $t\to \infty$, uniformly in $\theta$
in the following sense:  for any given $K\in\mathbb N$ and $\delta<1/2$,
 there exists a positive  constant $C$, independent of $\rho $ and $t$ such that
\begin{equation}
\label{q}
\sup_{\theta\in \mathbb R} (1+\theta^2)^K\,
|\mathcal N(\theta,t;\rho)|
\le C\rho \,t^{-\delta}
\mbox {   for any $t>2$}
\end{equation}
\end{lemma}
\begin{proof}
Consider the expression for $\log (\sqrt t)\,f(\theta\log (\sqrt t),t;\rho)$
 obtained from \eqref{f2},
 and  denote by $\mathcal N$ the term coming from the  first term on the r.h.s..
It is clear that it satisfies \eqref{q}.

The other term comes simply from organizing the fractions in the integral,
 \eqref{defb} and the formula $\vararch{y}=\log\big(y+\sqrt{y^2-1}\big)$. 
 \end{proof}
In the next lemmas, we show how to  obtain an explicit expansion in powers of $x$
 of the fractions 
in the last line in \eqref{lf}. Let us denote by $\Re(z)$ the real part of $z\in \bb Z$. 

\begin{lemma}
\label{gexp} Consider, for  $b,c,d \in \mathbb R$ such that 
$c^2-4d<0$, the rational function  
$$
g(x)=\frac{1+bx}{1+cx+dx^2}
$$ 
For any given natural $N\ge 0$, 
the following finite $N$--expansion with remainder holds:
$$
g(x)=
 \sum_{n=0}^{N} a_n x^n\,+\,q_N(x),
$$
\begin{align}
\label{an}
\mbox{ where }\quad &a_n=\frac {(-1)^n}{2^n}\,\Re
\big((1+i\frac{2b-c}{\sqrt{4d-c^2}})(c+i\sqrt{4d-c^2})^n\big)\quad n\ge -1
\\ \label{rn}
&q_N(x)=\,-x^{N+1}(\frac{da_{N-1}+ca_N+xda_N}{1+cx+dx^2})
\end{align}
\end{lemma} 
\begin{proof}
The sequence ${a_n}$ has to satisfy the recurrence
 $a_{n+2}+c\,a_{n+1}+d\,a_n=0$ for $n\ge 0$,  with $a_0=1$ and $a_1=b-c$. To solve it,  
let us introduce  the function 
\begin{equation}
\label{F}
F(x)=\sum_{n\ge 0}\frac{a_n\,x^n}{n!}
\end{equation}
From the recurrence for the ${a_n}'s$, it follows that $F$ has to satisfy the differential equation 
$F''+cF'+dF=0$, with boundary conditions  $F(0)=1, F'(0)=b-c$, whose solution is 
\begin{equation}
\label{Fexp}
F(x)=A\e^{rx}+\bar A\e^{\bar rx}
\end{equation}
with $A=
\frac12+i\frac{2b-c}{2\sqrt{4d-c^2}}$ and 
 $r=-\frac{c+i\sqrt{4d-c^2}}{2}$. 
Expanding the right hand side of 
\eqref{Fexp} in its  Taylor series, and equating coefficients with the expression
  \eqref{F}, we get   $a_n=2\Re(Ar^n)$, which is exactly \eqref{an}.

 The formula \eqref{rn} for the remainder is proven easily by induction for $N\ge 1$.
 Indeed, compute
$$
g(x)-a_0-(b-c)x=-x^2(\,\frac{d+c(b-c)+xd(b-c)}{1+cx+dx^2})
$$
to prove it  for $N=1$. Suppose it true for $N$, and write next 
\begin{equation}
\nn
g(x)-\sum_{n=0}^{N} a_n x^n
-a_{N+1}x^{N+1}=q_N(x)-a_{N+1}x^{N+1}.
\end{equation}
Use the recurrence relation for the $a_n$ to conclude that the r.h.s is $q_{N+1}$.
  
A direct computation shows that the formula is also valid for $N=0$, with  $a_{-1}$
given by taking $n=-1$ in \eqref{an}. 
 \end{proof}
 \medskip

\begin{lemma}
\label{fraclem}
The fractions on the second line of \eqref{lf} can be expanded
in powers of $x$ as follows: for any given integer 
$N\ge 0 $ and $x\le 4$, 
\begin{align}
\nonumber
&\frac{1+b\,x}{(1+b\,x)^2+(\theta+\frac{\pi}{2}\,x)^2}
\,+\,\frac{1+b\,x}{(1+b\,x)^2+(\theta-\frac{\pi}{2}\,x)^2}=
\\ \label{fcs}
 &\qquad \sum _{n=0}^{N}
\frac{(-1)^n\,P_n(b)\,\cos\big((n+1)\beta\big)}
{2^{n-1}(1+\theta^2)^{\frac {n+1}{2}}}\,x^n\,+\,Q_N,
\end{align}
where 
\begin{equation}
\label{beta}
\e^{i\beta}=\frac{1+i\theta}{(1+\theta^2)^\frac 12}, \quad 
P_n(b)=\!\sum_{k\le n, k\, even}\!
\binom{n}{k}(2b)^{n-k}(-1)^{\frac k2}\pi^k,
\end{equation}
\begin{align}
\nonumber 
\mbox{and }\quad  Q_N=&\frac{x^{N+1}}{(1+\theta^2)^{\frac N2 +1}}\,  \mathcal{E}_N (b,\theta)
\qquad \mbox{ with }  
\\
\label{QN}
|&\mathcal {E}_N (b,\theta)|\le 
K\,\frac{(4b^2+\pi ^2)^{\frac N2+1}}{2^{N-1}}\,
(\frac{\id_{|\theta|\le 20}}{(1+b\,x)^2}
\,+\,\id_{|\theta |> 20}) 
\end{align}
for some constant   $K>0 $ 
independent of $x$, $b,N$ and $\theta$. 
\end{lemma}

\begin{proof}

Let $N\ge 0$ given, and set   

\begin{equation}
\label{P1}
\frac{1+b\,x}{(1+b\,x)^2+(\theta+\frac{\pi}{2}\,x)^2}=
\frac{1}{1+\theta^2}\big(\frac{1+b\,x}{1+c\,x+d\,x^2}\big), 
\end{equation}
for  $c=\frac{2b+\pi\theta}{1+\theta^2}$  and   
$ d=\frac{b^2+\pi^2/4}{1+\theta^2}$. Observe that 
$c^2-4d=-\frac{(2b\theta-\pi)^2}{(1+\theta^2)^2}$, so  the conditions of
Lemma  \ref{gexp} are satisfied and we may apply  formulae
\eqref{an} and \eqref{rn} to obtain, after some algebraic manipulations, 
\begin{equation}
\label{F1} 
\frac{1+b\,x}{1+c\,x+d\,x^2}=
\sum_{n=0}^N 
a_n^{(1)}
\,x^n\,+\,q_N^{(1)}, \qquad \mbox{ with} 
\end{equation}
\begin {align}
\label{an1}
a_n^{(1)}\,&=\,\frac{(-1)^n} {2^{n}(1+\theta^2)^n}\Re
 \big((1+i\theta)^{n+1}(2b-i\pi)^n\big)\quad \mbox{ and}\\
\label{rn1}
q_N^{(1)}\,&=\,-x^{N+1}\,
\frac{(b^2+\pi^2/4)(a_{N-1}^{(1)}+\,x\,a_{N}^{(1)})
+(2b+\pi\theta)a_{N}^{(1)}}{1+\theta^2+(2b+\pi\theta)\,x+
(b^2+\pi^2/4)\,x^2}
\end{align}
Since the second fraction in \eqref{fcs} is just the first one evaluated at $-\theta$,
 the expansion for the former is: 
\begin{equation}
\label{F2}
\frac{1+b\,x}{(1+b\,x)^2+(\theta-\frac{\pi}{2}\,x)^2}=
 \frac{1}{1+\theta^2}\Big(\sum_{n=0}^N 
a_n^{(2)}
\,x^n\,+\,q_N^{(2)}\Big), 
\end{equation}
where the following expression for $a_n^{(2)}$ results from evaluating \eqref{an1} at $-\theta$ 
and conjugating the complex inside  $\Re$,
\begin{equation}
\nonumber
a_n^{(2)}\,=\,\frac{(-1)^n} {2^{n}(1+\theta^2)^n}\Re
 \big((1+i\theta)^{n+1}(2b+i\pi)^n\big) \mbox{ and }
q_N^{(2)}(\theta)\,=\,q_N^{(1)}(-\theta)
\end{equation}

From  \eqref{beta} and \eqref{an1} we have 
\begin{align}
\nonumber
a_n^{(1)}+a_n^{(2)}&= \frac{(-1)^n}{2^n(1+\theta^2)^n}\,
\Re \Big((1+i\theta)^{n+1}\big((2b-i\pi)^{n}+(2b+i\pi)^{n}\big)\Big)
\\ \nonumber
&=\frac{(-1)^n\,P_n(b)}{2^{n-1}(1+\theta^2)^n}\,
\Re\Big((1+i\theta)^{n+1}\Big)=
\frac{(-1)^n\,P_n(b)}{2^{n-1}(1+\theta^2)^{\frac{n-1}{2}}}
\cos\big((n+1)\beta\big).
\end{align}
Then, \eqref{fcs} follows from \eqref{P1}, \eqref{F1} 
and \eqref{F2},  with $Q_N=\frac{1}{1+\theta^2}(q_N^{(1)}+q_N^{(2)})$. 
To complete the proof, we only need to show 
\eqref{QN}. From \eqref{an1},
\begin{equation}
\nonumber
|a_n^{(1)}|\le 
\frac{\big |(1+i\theta)^{n+1}(2b-i\pi)^n\big |}{2^n(1+\theta^2)^n}\,=\,
 \frac{(4b^2+\pi^2)^{\frac n2}}{2^n(1+\theta^2)^{\frac {n-1}{2}}},
\end{equation}
and then from  \eqref{rn1}, there is a constant  $K_1>0$ such that 
\begin{equation}
\label{arn}
|q_N^{(1)}|\le\,x^{N+1}\,K_1\, \frac{(4b^2+\pi^2)^{\frac N2+1}}
{D\,2^{N-1}(1+\theta^2)^{\frac N2-1}},
\end{equation}
where $D$ is the denominator in \eqref{rn1}.
 Recall that $x \le 4$ to estimate
\begin{align}
\nonumber
D&= 1+\theta^2+(2b+\pi\theta)\,x+
(b^2+\pi^2/4)\,x^2=(1+\theta^2)
\Big(\frac{(1+b\,x)^2}{1+\theta^2}
+\frac{(\theta+\frac\pi 2\eta)^2}{1+\theta^2}\Big)
\\ \nonumber
&\ge(1+\theta^2)\,\Big(\frac{(1+b\,x)^2}{1+\theta^2}
\id_{|\theta|\le 20}+
\frac{(\theta+\frac\pi 2\,x)^2}{1+\theta^2}
\id_{|\theta|> 20}\Big)
\\ 
& \ge (1+\theta^2)\Big(\frac{(1+b\,x)^2}{401}
\id_{|\theta|\le 20}+
\frac 13
\id_{|\theta|> 20}\Big),
\end{align}
so, from \eqref{arn},
$$ 
|q_N^{(1)}|\le K_2\,x^{N+1}\,\frac{(4b^2+\pi^2)^{\frac N2+1}}
{2^{N-1}(1+\theta^2)^{\frac N2}} \,\big(\frac{\id_{|\theta|\le 20}}{(1+b\,x)^2}
\,+\,\id_{|\theta |> 20}\big) 
$$

The same estimate holds for $|q_N^{(2)}|$, and  \eqref{QN} follows. 
\end{proof}

\noindent {\it Proof of Theorem \ref{th1}}

The proof consists in substituting the fractions in  \eqref{lf} by the
$N$--expansion \eqref{fcs} and then integrate term by term, recalling that only $b$
depends on the integration variable $z$. 
Let  $N\ge 0$ be given, recall that $x=\frac{1}{\log\sqrt t}$, and write then
\begin{multline}
\log (\sqrt t)\,f(\theta\log (\sqrt t),t;\rho)
\,=\,\mathcal N(\theta,t;\rho)\,+
\\
\label{lfff}
\frac{1}{2\pi\sqrt\pi}\,\sum _{n=0}^{N}(-1)^n \,x^n\, 
\frac{\,\cos\big((n+1)\beta\big)}
{2^{n-1}(1+\theta^2)^{\frac {n+1}{2}}}
\int ^{\infty}_{\rho^2/4t}\,
\frac{\e^{-(z+\rho^2/4t)}}{\sqrt{z+\rho^2/4t}}\,P_n(b)\, dz\,+\\
\frac{1}{2\pi\sqrt\pi}\,\int ^{\infty}_{\rho^2/4t}\,
\frac{\e^{-(z+\rho^2/4t)}}{\sqrt{z+\rho^2/4t}}\,Q_N \,dz.\quad \quad
\end{multline}
Call 
\begin{align}
\label{cnf}
&C_n=\int ^{\infty}_{\rho^2/4t}\,
\frac{\e^{-(z+\rho^2/4t)}}{\sqrt{z+\rho^2/4t}}\,P_n(b)\, dz
\\
\label{qnf}
&R_N=\mathcal N(\theta,t;\rho)+\frac{1}{2\pi\sqrt\pi}\,\int ^{\infty}_{\rho^2/4t}\,
\frac{\e^{-(z+\rho^2/4t)}}{\sqrt{z+\rho^2/4t}}\,Q_N \,dz.
\end{align}
From \eqref{beta}, $\cos\big((n+1)\beta\big)=A_n(\theta)$, so to prove Theorem \ref{th1}
it is enough to show that $C_n$ and $R_N$ satisfy \eqref{cnt} and  \eqref{RN} respectively. 

Recall the definition of $b$ in  \eqref{defb}, and  set 
\begin{equation}
\label{defb0}
b_0=\underset{t\to\infty}{\lim}\,b\,=\frac12 \log (\frac{8 z}{\rho^2}) \quad \Rightarrow 
c_n=\underset{t\to\infty}{\lim} C_n =\int ^{\infty}_{0}\,\frac{\e^{-z}}{\sqrt{z}}\,P_n(b_0)\,dz
\end{equation}
as follows from \eqref{beta} and \eqref{cn}. The ${c_n}'s$ are clearly finite, and we have  
\begin{equation}
\label{Cnest}
C_n=c_n
+
\int ^{\infty}_{\rho^2/4t}\frac{P_n(b)\,\e^{-(z+\rho^2/4t)}}{\sqrt{z+\rho^2/4t}}
\,-\frac{P_n(b_0)\,\e^{-z}}{\sqrt{z}}\,dz\,-\int _{0}^{\rho^2/4t}
\frac{P_n(b_0)\,\e^{-z}}{\sqrt{z}}\,dz
\end{equation}

To estimate the integrals above, let us  shorthand $w=\frac12 \big(1-\sqrt{1-(\frac{\rho^2}{4zt})^2}\big)$
 and note that $0\le w\le \rho^2/8tz\le 1/2$ for $z\ge\rho^2/4t$  to obtain
\begin{equation}
\label{bb0}
|b-b_0|= -\frac12 \log\, (1-w)=\frac12 \log\, (1+\frac{w}{1-w})\le w\le  \frac{\rho^2}{8zt}
\end{equation}
Denote by $k_n$ positive constants depending only on $n$, that may change from line to line. 
From \eqref{beta} and the previous inequality, it follows that
\begin{equation}
\label{dPb}
|P_n(b)-P_n(b_0)|\le k_n\,|b-b_0|\big(|b_0|^{n-1} +1\big),
\end{equation}
Then, let $0<\delta_0<1/2$, and estimate with the aid of \eqref{bb0} and \eqref{dPb}, 
\begin{multline}
 \label{rot}
 \big|\int ^{\infty}_{\rho^2/4t}\,\frac{\e^{-(z+\rho^2/4t)}}{\sqrt{z+\rho^2/4t}}
\,P_n(b)-\frac{\e^{-z}}{\sqrt{z}}\,P_n(b_0)\,dz\big|\,\le
\\
\int ^{\infty}_{\rho^2/4t}\frac{\e^{-(z+\rho^2/4t)}}{\sqrt{z+\rho^2/4t}}
\,|P_n(b)-P_n(b_0)| dz\, + \phantom{1234567890}
\\
\int ^{\infty}_{\rho^2/4t}\frac{\e^{-z}}{\sqrt{z}}
\,|P_n(b_0)|\big(1-\frac{\e^{-\rho^2/4t}\sqrt z}{\sqrt{z+\rho^2/4t}}\big)\,dz
\\
\le \frac{k_n}{t} \int ^{\infty}_{\rho^2/4t}\,\e^{-z}z^{-3/2}\big(|b_0|^{n-1} +1\big)dz\,+\,
 \int ^{\infty}_{\rho^2/4t}\frac{\e^{-z}}{\sqrt{z}}\,|P_n(b_0)|
 \big(\frac{\rho^2}{4t}+\frac{\rho}{\sqrt{4tz}}\big)dz\\
 \le \frac{k_n}{t}\big(\int ^{1}_{\rho^2/4t}\!z^{-(3/2)-\delta_0}dz\,+1\big)+
 \frac{k_n}{t}\big(\int ^{1}_{\rho^2/4t}z^{-(1/2)-\delta_0}dz\,+1\big)
 \\
 +\,\frac{k_n}{\sqrt t}\big(\int ^{1}_{\rho^2/4t}z^{-1-\delta_0}dz\,+1\big)
 \le k_n\,t^{-(1/2)+\delta_0}
 \end{multline}
To obtain the last inequality, we have splitted the integrals according to $z<1$ or $z\ge1$, 
after observing that,  for  $S_n$ a polynomial of degree $n$,  $|S_n(b_0)|\,z^{\delta_0}\le k_n$ for $0<z<1$ . 
For $z\ge 1$, the integrals are clearly uniformly bounded. 
The estimation of the  last term in \eqref{Cnest} is  similar: 
\begin{align}
\label{ini}
\big|\int _{0}^{\rho^2/4t}\!
\frac{\e^{-z}}{\sqrt{z}}\,P_n(b_0)\,dz\big|
\le  k_n t^{{-\frac 12+\delta_0}},
\end{align} 
 which together with \eqref{rot} and \eqref{Cnest} concludes the proof of \eqref{cnt}.
 
It remains to show that $R_N$ satisfies \eqref{RN}. Consider first the case 
$|\theta| > 20$. From \eqref{q} and  \eqref{QN}, it is enough in this case  to show that 
\begin{equation}
\label{rest}
\int ^{\infty}_{\rho^2/4t}\,
\frac{\e^{-(z+\rho^2/4t)}}{\sqrt{z+\rho^2/4t}}\frac{(4b^2+\pi ^2)^{\frac N2+1}}{2^{N-1}}dz\le k_N,
\end{equation}
which can be seen by the same procedure as that to estimate 
 \eqref{cnf} (just considering $(4b^2+\pi ^2)^{\frac N2+1} $
instead of $P_n(b)$). 

In the case  $|\theta| \le 20$, 
let us split  the integral in \eqref{qnf} according to $z\in (\rho^2/4t\,,\,\rho^2/4\sqrt t)$, or
$z>\rho^2/4\sqrt t $.
For the latter, observe  that 
\begin{equation}
\label{1bx}
1+bx=\frac{\log\big(4zt/\rho^2+\sqrt{(4zt/\rho^2)^2-1}\big)}{\log {t}}
\ge \frac{\log{\sqrt t}}{\log {t}}= \frac12
\end{equation}
so from   \eqref{QN}, again  a procedure as that to estimate the integral in \eqref{cnf} 
yields that this integral is bounded. 
For the case $z\in (\rho^2/4t\,,\,\rho^2/4\sqrt t)$, 
use the expression for $Q_N$  coming  from \eqref{fcs} and take $0<\delta_1<\frac 14$
\begin{multline}
\label{zpeq}
\Big|\int ^{\rho^2/4 \sqrt t}_{\rho^2/4t}\,
\frac{\e^{-(z+\rho^2/4t)}}{\sqrt{z+\rho^2/4t}}\times
\Big(\frac{1+b\,x}{(1+b\,x)^2+(\theta+\frac{\pi}{2}\,x)^2}
\,+\,\frac{1+b\,x}{(1+b\,x)^2+(\theta-\frac{\pi}{2}\,x)^2}
\\ 
 -\sum _{n=0}^{N}
\frac{(-1)^n\,P_n(b)\,\cos\big((n+1)\beta\big)}
{2^{n-1}(1+\theta^2)^{\frac {n+1}{2}}}\,x^n \Big)\, dz\Big|
\\
 \le 
\int ^{\rho^2/4 \sqrt t}_{\rho^2/4t}
\frac{\e^{-(z+\rho^2/4t)}}{\sqrt{z+\rho^2/4t}}
\big(\frac{2}{1+b\,x}\big)\,dz+\,k_N
\int ^{\rho^2/4 \sqrt t}_{\rho^2/4t}
\frac{\e^{-(z+\rho^2/4t)}}{\sqrt{z+\rho^2/4t}}\sum _{n=0}^{N}
|P_n(b)|dz
\\
\le t^{-\frac 14 +\delta_1}
\end{multline}

To obtain the last inequality above, start by estimating the first integral in the previous line. Observe that
 $1+b\,x=\frac{\vararch\, (4zt/\rho^2)}{\log t}$, and change variables $y=\vararch(\frac{4zt}{\rho^2})$
 to obtain:  
\begin{multline}
\label{up}
\int ^{\rho^2/4 \sqrt t}_{\rho^2/4t}\,
\frac{\e^{-(z+\rho^2/4t)}}{\sqrt{z+\rho^2/4t}}
\frac{1}{(1+b\,x)}\,dz
\le 
\log t\int ^{\rho^2/4 \sqrt t}_{\rho^2/4t}\,
\frac{1}{\sqrt{z}\,\vararch\, (\frac{4zt}{\rho^2})}dz
\\
=\frac{\rho\,\log t}{\sqrt t}\big(\int_0^1 \frac{\varsh y}{y\sqrt{\varch y}}\,dy\,+\,
\int_1^{\vararch(\sqrt t)} \frac{\varsh y}{y\sqrt{\varch y}}\,dy\big)\le
\frac{\rho\,\log t}{\sqrt t}\big(k_1+k_2\,t^{\frac 14}\big) 
\end{multline}
The last integral in \eqref{zpeq} is bounded by
$t^{-\frac 14 +\delta_1}$, as can be seen from estimates similar to those used to estimate
\eqref{cnf},  what finishes the proof of 
Theorem \ref{th1}. 
\qed 

\bigskip

\noindent{\it Proof of Theorem \ref{th2}}\quad
Recall that $\rho=1$ and $x=1/\log\sqrt t$. From \eqref{f2} and \eqref{defb}, we have 
\begin{align}
\nonumber
&\log (\sqrt t)\,f(\theta,t)
\,=
\\
\nonumber
&\quad \frac{\log (\sqrt t)}{\sqrt{2\pi t}}\,\e^{-\frac{1}{4 t}(1-\cos 2\theta)}
         \cos{\theta}\,
         \id _{(-\pi/2,\pi/2)}(\theta)\,+\,
 \frac{1}{2\pi\sqrt\pi}\,\int ^{\infty}_{1/4t}\,
\frac{\e^{-(z+1/4t)}}{\sqrt{z+1/4t}}
\\
\label{lff}
&\times\,\Big(\frac{1+bx}{(1+bx)^2+(\theta+\frac{\pi}{2})^2x^2}
\,+\,\frac{1+bx}{(1+bx)^2+(\theta-\frac{\pi}{2})^2x^2}
\Big)\, dz.
\end{align}
Let us denote by $k$ positive constants that may change from line to line, 
and split the integral above according to $z\in (1/4t,1/4\sqrt t)$ or $z> 1/4\sqrt t$. 
Proceeding as in the estimation of  \eqref{up},  it follows that
\begin{align}
\nonumber
&\frac{\log (\sqrt t)}{\sqrt{2\pi t}}\,\e^{-\frac{1}{4 t}(1-\cos 2\theta)}
         \cos{\theta}\,
         \id _{(-\pi/2,\pi/2)}(\theta)\,+\,
 \frac{1}{2\pi\sqrt\pi}\,\int ^{1/4\sqrt t}_{1/4t}\,
\frac{\e^{-(z+1/4t)}}{\sqrt{z+1/4t}}\,\times 
\\ 
\label{K1}
&\Big(\frac{1+bx}{(1+bx)^2+(\theta+\frac{\pi}{2})^2x^2}
\,+\,\frac{1+bx}{(1+bx)^2+(\theta-\frac{\pi}{2})^2x^2}
\Big)\, dz\le k\,t^{-\frac 14}
\end{align} 

 Expanding the fractions in \eqref{lff} with the aid of Lemma \ref{gexp}, 
we have
\begin{equation} 
\label{F3}
\frac{1+bx}{(1+bx)^2+(\theta\pm\frac{\pi}{2})^2x^2}=\sum_{n=0}^N x^nA_n(b,\pm\theta)\,+T_N^{\pm}
\end{equation}
where  $T_N^{-}$ is  $T_N^{+}$ evaluated at  $-\theta$, and, for   $d=b^2+(\theta+\frac \pi 2)^2$, 
\begin{align}
\label{An}
&A_n(b,\theta)=(-1)^n\Re \big(b+i(\frac \pi 2+\theta)\big)^n
\\ \label{Tn}
&T_N^{+}=-x^{N+1}\,\frac{dA_{N-1}+2bA_N+xdA_N}{(1+bx)^2+(\theta+\frac{\pi}{2})^2x^2}
\end{align}
From  \eqref{An}, it is easy to see that 
\begin{equation}
\label{AAn}
|A_n(b,\pm\theta)|\le \big(b^2+(\frac \pi 2\pm\theta)^2\big)^{\frac n2}. 
\end{equation}
For   $z>\frac{1}{1/4\sqrt t}$ the denominators in $T_N^{\pm}$ are 
$>(\frac 12)^2$, as follows from \eqref{1bx}. Therefore
\begin{equation}
\nonumber
%\label{estT}
|T_N^{\pm}|\le x^{N+1}\,k\,\big(b^2+(\frac \pi 2\pm\theta)^2\big)^{\frac N2 +1},
\end{equation}
 and proceeding as in the estimation of \eqref{rest}, we obtain 
\begin{align}
\label{itn}
 &\frac{1}{2\pi\sqrt\pi}\,\int ^{\infty}_{1/4\sqrt t}\,
\frac{\e^{-(z+1/4t)}}{\sqrt{z+1/4t}}\,\big(
T_N^{-}+T_N^{+}\big)\, dz\le k_N \frac{1+|\theta |^{N+2}}{(\log \sqrt t)^{N+1}}
\end{align}
On the other hand,  from \eqref{An}, 
it is easy to see that 
\begin{equation}
\label{dAn}
|A_n(b,\pm \theta)-A_n(b_0,\pm \theta)|\le |b-b_0|\,k_n\,\big(|b_0|^n+|\theta|^n\,+\,1\big)
\end{equation}

Next, set 
\begin{equation}
\nonumber
g_n=\int ^{\infty}_{0}\,
\frac{\e^{-z}}{\sqrt{z}}\, \big((A_n(b_0,\theta)\,+A_n(b_0,-\theta)\big)\,dz,
\end{equation}
which from  \eqref{An} is precisely \eqref{gn}. Note that the ${g_n}'s$
 are polynomials in $\theta$ of degree $n$. 
To finish the proof, from \eqref{lff}, \eqref{K1}, \eqref{F3} and \eqref{itn}
it is enough  to show that for $\delta_1<\frac 14$, 
\begin{multline}
\int ^{\infty}_{1/4\sqrt t}\frac{\big(A_n(b,\theta)+A_n(b,-\theta)\big)\,\e^{-(z+1/4t)}}{\sqrt{z+1/4t}}
\,-\frac{\big(A_n(b_0,\theta)+A_n(b_0,-\theta)\big)\,\e^{-z}}{\sqrt{z}}\,dz\,
\\
-\int _{0}^{\rho^2/4t}
\frac{\big(A_n(b_0,\theta)+A_n(b_0,-\theta)\big)\,\e^{-z}}{\sqrt{z}}\,dz\le k_n 
(1+|\theta|^n)t^{-\frac14 +\delta_1},
\end{multline} 
what follows from estimates similar to those in \eqref{rot} and \eqref{ini}, with the aid of \eqref{AAn} and \eqref{dAn}. 
\qed

\bigskip
 
\bibliography{winda}{}

\begin{small}
Stella Brassesco $^{(1)}$ \hskip 2cm  Silvana C. Garc\'\i a Pire $^{(2)}$ 
\\ \texttt{sbrasses@ivic.gob.ve}\hskip 1.5cm \texttt{sgarcia@ivic.gob.ve}
\\ $^{(1)}$ Departamento de Matem\'aticas,
    Instituto Venezolano de Investigaciones Cient\'{\i}ficas, Apartado Postal
    20632 Caracas
    1020--A, Venezuela
    \\$^{(2)}$ Universidad Nacional Experimental Sim\'on Rodriguez,  N\'ucleo Araure
    Av. 13 de Junio con calle 5, Edif. Araure, Estado Portuguesa, Venezuela. 
    \end{small}

\end{document}